\documentclass{birkjour}
\usepackage{amssymb}
\newtheorem{lemma}{Lemma}
\newtheorem{theorem}{Theorem}
\newtheorem{atheo}{Theorem}

\newtheorem{definition}{Definition}
\newtheorem{proposition}{Proposition}
\newtheorem{remark}{Remark}

\newcommand{\R}{\mathbb{R}}

\newcommand{\ds}{\displaystyle}

\begin{document}

%
%
%
%
%
%
%
%
%

\title[Curved fronts in a shear flow: case of combustion nonlinearities]
 {Curved fronts in a shear flow: case of combustion nonlinearities}
\date{}

\author[M. El Smaily]{Mohammad El Smaily \footnote{M. El Smaily is supported in part by the Natural Sciences and Engineering Research Council of Canada through  NSERC Discovery Grant RGPIN-2017-04313}}

\address{
Department of Mathematics \& Statistics, \br University of New Brunswick,\br Fredericton, NB, E3B5A3, Canada.}

\email{m.elsmaily@unb.ca}



\date{}
\maketitle
\begin{abstract} We prove the existence and uniqueness, up to a shift in time, of curved traveling fronts for a reaction-advection-diffusion equation with a combustion-type nonlinearity. The advection is through a shear flow $q$. This analyzes, for instance, the shape of flames produced by a Bunsen burner in the presence of advection. We also give a formula for the speed of propagation of these conical fronts in terms of the well-known speed of planar pulsating traveling waves.\hfill\break
\end{abstract}

\section{Introduction and main results}
This paper is concerned with the existence, uniqueness and qualitative properties of curved traveling waves solutions to the reaction-advection-diffusion problem  \begin{equation}\label{ueq}
\partial_{t}u(t,x,y)=\Delta_{x,y} u + q(x)\partial_y u(t,x,y)+f(u) \qquad \hbox{for all $t \in \mathbb{R}, \,(x,y)\in \mathbb{R}^2$} 
\end{equation} 
and satisfy certain limiting properties as the vertical direction $y$ goes to $\pm\infty.$ The advection coefficient $x\mapsto q(x)$ belongs to $C^{1,\delta}(\mathbb R)$ for some $\delta>0$ and satisfies the periodicity and normalization conditions
\begin{equation}\label{cq}
\forall\,x\in\mathbb{R},\quad q(x+L)=q(x) \quad \hbox{and} \quad \displaystyle{\int_{0}^L q(x)\;dx=0} \hbox{ for some }L>0. 
\end{equation}
Thus, the advection field  $\tilde q(x,y)=(0,q(x))$ is divergence free and is of ``shear-flow'' type.  

\noindent The function $f$ is Lipschitz-continuous in [0,1], continuously differentiable in a left neighbourhood $(1-r,1]$ of  1 and satisfies 
\begin{equation}\label{cf}
\exists\, \theta\in (0,1); ~f\equiv0 \hbox { on }[0,\theta]\cup\{1\},~ f>0 \hbox{ on }(\theta,1) \hbox{ and }f'(1)<0.
\end{equation}
We extend $f$ by 0 outside $[0,1]$. Hence, $f$ is Lipschitz-continuous on $\R$. From standard elliptic estimates, any bounded solution $u$ of \eqref{ueq} is of class $C^{2,\delta}(\R^{2})$ for any $\delta\in[0,1)$. We will often refer to this class of functions as ``{\em combustion-type}'' nonlinearities and the parameter $\theta$ is to stand for  the {\em ignition-temperature}.

In this work, we  are interested in solutions of (\ref{ueq}) that are curved traveling fronts  which have the form
$$u(t,x,y)=\phi(x,y+ct)$$
for all $(t,x,y)\in\mathbb R\times\mathbb R^2$, and for some positive constant $c$ which denotes the speed of propagation in the vertical direction~$-y$. Thus, we are led to the following elliptic equation
\begin{equation}\label{cequation}
\Delta\phi+(q(x)-c)\partial_{y}\phi+f(\phi)=0 \;\hbox{ for all }\;(x,y)\in\mathbb{R}^2.
\end{equation}
The word ``curved'' appearing in the name of  these solutions comes from the requirement that they satisfy the following conical limiting conditions
\begin{equation}\label{concon}
\displaystyle{\lim_{l\rightarrow-\infty}}\Big(\displaystyle{\sup_{(x,y)\in C^{-}_{\alpha,l}}}\phi(x,y)\Big)= 0\quad \text{and} \quad 
\displaystyle{\lim_{l\rightarrow\infty}\Big(\inf_{(x,y)\in C^{+}_{\alpha,l}}\phi(x,y)\Big)= 1,}\end{equation}
where $\alpha$  is given in $(0,\pi)$ and the lower and upper cones~$C^{-}_{\alpha,l}$ and~$C^{+}_{\alpha,l}$ are defined as follows:
\begin{definition}\label{cones} Let $\alpha \in (0,\pi).$ For every real number $l$, the lower cone $C^{-}_{\alpha,l}$ is defined by
$$\begin{array}{cl}
&C^{-}_{\alpha,l}=\big\{(x,y)\in\mathbb{R}^2,~~ y\leq x\cot\alpha+l~\hbox{ whenever }~ x\leq0\vspace{3pt}\\
&\hskip3cm\hbox{ and }~ y\leq-x\cot\alpha+l ~\hbox{ whenever }~ x\geq0\big\}\end{array}$$
and then the upper cone $C^{+}_{\alpha,l}$ is defined by
$$C^{+}_{\alpha,l}=\overline{\mathbb{R}^2 \setminus C^{-}_{\alpha,l}}\,.$$
\end{definition}
Before we go further, let us explain briefly why would one be interested in such curved-fronts. Equation \eqref{cequation} or its equivalent parabolic version \eqref{ueq} arise  in models of equi-diffusional premixed Bunsen flames, for instance. The function $u$ or $\phi$ represents a normalized temperature and its level sets represent the conical-shaped flame coming out of the Bunsen burner. The temperature of the unburnt gases is close to 0 and that of burnt gases is close to 1.  The real number $c$ can be interpreted as the speed of the gas at the exit of the burner (see the works \cite{Sivashinsky1} and \cite{Sivashinsky} by Sivashinsky and \cite{Williams} by Williams).  
 \subsection{Prior works}
Several works have considered the problem of conical fronts in various settings. Bu and Wang \cite{BuWangZhi, BuWang2}  consider the problem in 3 dimensions, in presence of a combustion-type nonlinearity, but without an advection term. They prove existence, uniqueness and asymptotic stability of three-dimensional pyramidal traveling fronts under certain conditions. In another work, Bu and Wang \cite{BuWang1} consider the problem in presence of advection, but with KPP-type nonlinearities (in contrast with combustion nonlinearities that we consider here). In \cite{BuWang1}, the authors generalize the results of \cite{ehh} to higher dimensions by proving existence of pyramidal fronts in dimensions 3 and 4.  Curved fronts were also studied in the case of bistable nonlinearities, though  without an advection term, in the works \cite{Taniguchi2005} and  \cite{Taniguchi2006} by Taniguchi and Ninomiya.  The authors of  \cite{Taniguchi2005} and  \cite{Taniguchi2006} studied the existence and stability of travelling curve fronts to the Allen-Cahn equation.  One of the earliest works on the conical-fronts question was that by Bonnet and Hamel \cite{BonnetHamel} and Hamel, Monneau \cite{HMonneau}. The results of \cite{BonnetHamel} were later generalized, to any dimension $N,$ by Roquejoffre, Hamel and Monneau \cite{HamelMonneauRoquejoffre04} which proved  the existence, and the global stability, of travelling waves solutions with conical- shaped level sets. The authors of \cite{HamelMonneauRoquejoffre04} also studied the same type of questions but for a bistable nonlinearity, instead of combustion-type nonlinearity, in the later work \cite{HamelMonneauRoquejoffre05}.

\subsection{Auxiliary problem:  pulsating fronts propagating to the left and to the right}
We start by recalling some known results about planar traveling fronts in the case of ignition nonlinearity of type \eqref{cf}. For each positive definite symmetric matrix $M,$ consider the following problem whose solutions are planar traveling fronts connecting $0$ to $1$:

 \begin{equation}\label{aux-u-eq}\begin{array}{rcl}
\frac{\partial u}{\partial t} & \!\!=\!\! & {\rm div}(M \nabla u)+q(X)\sin\gamma \frac{\partial u}{\partial Y}\!+\!f(u),\ \, t\in \mathbb R,\; (X,Y)\!\in\mathbb R^2,\vspace{3pt}\\
u(t\!+\!\tau,X\!+\!L,Y) & \!\!=\!\! & u(t\!+\!\tau,X,Y)=u(t,X,Y\!+\!c\tau),\ \,(t,\tau,X,Y)\!\in\mathbb R^2\!\times\!\mathbb R^2,\vspace{3pt}\\
u(t,X,Y) & \!\!\!\!\underset{Y\rightarrow -\infty}{\longrightarrow}\!\!\!\! &  0,\quad u(t,X,Y)\underset{Y\rightarrow \infty}{\longrightarrow} 1,\end{array}
\end{equation}
Note that the limiting conditions at $\pm\infty$ in \eqref{aux-u-eq} are not ``conical''. Moreover, the drift term depends only on the $X$ variable and the reaction term $f$ does not depend on the space variables. Thus, the ansatz $u(t,X,Y)=\varphi(X,Y+ct)$ requires that the pair $(c,\varphi)$ solves the following problem
\begin{equation}\label{aux-phi-eq}\begin{cases}
\hbox{div}(M \nabla \varphi)+(q(X)\sin\gamma-c){\partial_Y\varphi}+f(\varphi)=0, \quad (X,Y)\in\mathbb R^2,\\
\varphi(X,Y)\underset{Y\rightarrow -\infty}{\longrightarrow} 0,\quad \varphi(X,Y)\underset{Y\rightarrow \infty}{\longrightarrow} 1,\quad \hbox{uniformly in }X\in\mathbb R,\\
\varphi(X+L,Y)=\varphi(X,Y), \quad (X,Y)\in \mathbb R^2.\end{cases}
\end{equation}
A solution $(c,\varphi)$ of \eqref{aux-phi-eq} is known   as a {\it pulsating traveling front} in the vertical direction and the constant $c$ represents the speed. We recall the existence and uniqueness theorem of pulsating traveling fronts which follows from a more general result by Berestycki and Hamel \cite{bh02}:
\begin{atheo}[Berestycki, Hamel \cite{bh02}]\label{old.planar} If $q$ and $f$ satisfy \eqref{cq} and \eqref{cf}, then \eqref{aux-u-eq} or equivalently \eqref{aux-phi-eq} admits a pulsating traveling front $u(t,X,Y)=\varphi(X,Y+ct)$ and a unique speed of propagation $c=c_{M,q\sin\alpha,f}$. Furthermore, the traveling front solution $u$ is unique up to shifts in the time variable $t.$
\end{atheo}
\noindent We mention that a variational min-max formula for the unique speed $c_{M,q\sin\alpha,f}$ of pulsating traveling fronts in the case of combustion nonlinearity is derived in El Smaily \cite{min-max}. Furthermore, the asymptotic behaviour of this speed in presence of a shear-flow drift term with a large amplitude has been studied in Hamel and Zlato\v{s} \cite{HamelZlatos}.  

 In what follows, we will use the diffusion matrices 
 \begin{equation}\label{matrixAandB}
 A=\left[
\begin{array}{cc}
 1 &   \cos \alpha \\
  \cos\alpha&  1   
\end{array}
\right] \hbox{ and } B=\left[
\begin{array}{cc}
 1 &   -\cos \alpha \\
  -\cos\alpha&  1   
\end{array}
\right].
\end{equation}

\noindent The following proposition clarifies the role of the symmetry assumption we placed on the advection term $q$. This in turn will allow us to construct a sub and supersolution which consist of the right and left moving fronts for our main problem \eqref{ueq} coupled with conditions \eqref{concon}.

\begin{proposition}[On the symmetry assumption $q(x)\equiv q(-x)$]
Suppose that $q(x)=q(-x)$ for all $x\in\mathbb R$. Then, in the above notation, we have
$c_{A,q\sin\alpha,f}=c_{B,q\sin\alpha,f}.$
\end{proposition}
\begin{proof}
Let $(c_{A,q\sin\alpha,f}, \varphi(X,Y))$ be the unique solution of the {\emph pulsating traveling front} problem
\begin{equation}\label{aux1}\begin{cases}
\hbox{\rm div}(A \nabla\varphi(X,Y))+(q(X)\sin\alpha-c_{A,q\sin\alpha,f}){\partial_Y\varphi(X,Y)}+f(\varphi)=0\hbox{ in }\mathbb R^2,\!\!\vspace{7pt}\\
\varphi(X,Y)\underset{Y\rightarrow -\infty}{\longrightarrow} 0,\quad\varphi(X,Y)\underset{Y\rightarrow \infty}{\longrightarrow} 1 \hbox{ uniformly in }X\in\mathbb R.\end{cases}
\end{equation}
Note that \eqref{aux1} is the corresponding equation to \eqref{aux-phi-eq} where $M$ is replaced by the matrix $A$. Then
define $\psi(X,Y):=\varphi(-X,Y)$ for all $(X,Y)\in\mathbb R^2$. Since $q(X)=q(-X)$ for all $X\in\mathbb R,$  the pair $(c_{A,q\sin\alpha,f}, \psi)$ is then a solution of the following problem
\begin{equation}\label{aux2}\begin{cases}
\hbox{\rm div}(B \nabla \psi(X,Y))+(q(X)\sin\alpha-c_{A,q\sin\alpha,f})\,{\partial_Y \psi(X,Y)}+f(\psi)=0\hbox{ in }\mathbb R^2,\vspace{7pt}\\
\psi(X,Y)\underset{Y\rightarrow -\infty}{\longrightarrow} 0,\quad \psi(X,Y)\underset{Y\rightarrow \infty}{\longrightarrow} 1 \hbox{ uniformly in }X\in\mathbb R.\end{cases}
\end{equation}
However, a solution of \eqref{aux2} is a pulsating traveling front corresponding to the diffusion matrix $B$ and propagating in the direction of $-e=(0,-1).$ 
As the reaction $f$ is of combustion type, we know  from \cite{bh02}, Theorem \ref{old.planar} above, that  problem (\ref{aux2}) admits a unique speed of propagation \emph{which we denoted  above by} $c_{B,q\sin\alpha,f}$ (\emph{\cite{bh02} also proves that the solutions 
$v(t,X,y):=\psi(X,Y+c_{B,q\sin\alpha,f}t)$ of the parabolic equation \[v_{t}=\nabla\cdot(B\nabla v)+q(X)\sin\alpha\partial_{Y}v+f(v),\] with the limiting conditions $\ds{\lim_{Y\rightarrow+\infty}v(t,X,Y)=1}$ and $\ds{\lim_{Y\rightarrow-\infty}v(t,X,Y)=0}$, are unique up to a shift in $t$}). Therefore the condition $q(x)=q(-x)$ leads to $c_{A,q(x)\sin\alpha,f}=c_{B,q(x)\sin\alpha,f}.$ \end{proof}

\subsection{Statement of main results}

\begin{theorem}[Existence and uniqueness]\label{existence}
Let $\alpha\in (0,\pi).$ Under the  assumptions \eqref{cq} and \eqref{cf} on $q$ and $f,$ there exists a unique speed $c$ and a solution $u(t,x,y)$ of the form $u(t,x,y)=\phi(x,y+ct)$ of  equation \eqref{ueq} which satisfies the conical conditions (\ref{concon}). Moreover, the curved traveling front  $u$ is unique up to a shift in $t$ and the speed $c$ is given by the formula
\begin{equation}\label{formula.of.c}
\displaystyle{c=\frac{c_{A,q\sin\alpha,f}}{\sin\alpha}=\frac{c_{B,q\sin\alpha,f}}{\sin\alpha},}
\end{equation}
where ${c_{A,q\sin\alpha,f}}$ is the unique speed of pulsating traveling fronts for the auxiliary problem \eqref{aux1}.
In other words, if $(c_{1}, u_{1})$ and $(c_{2},u_{2})$, with 
$u_{1}(t,x,y)=\Phi_{1}(x,y+c_{1}t)$  and $u_{2}(t,x,y)=\Phi_{2}(x,y+c_{2}t),$  solve \eqref{cequation} with the conical conditions \eqref{concon} on $\Phi_{1}$ and $\Phi_{2},$ then $c_{1}=c_{2}=c$ {\rm(}given in \eqref{formula.of.c}{\rm)} and $u_{1}(t,x,y)=u_{2}(t+\kappa,x,y)$ for some $\kappa\in\R.$ 
\end{theorem}

\begin{theorem}[Monotonicity]\label{monotonicity}
The conical front $\phi=\phi(x,y)$ which solves \eqref{cequation} and satisfies the limiting conditions \eqref{concon} is increasing in the $y$ variable.
\end{theorem}
\begin{remark}[Differences between  `combustion' and  `KPP' nonlinearities] {\rm We comment on the influence of the nonlinearity $f$ on the problem by recalling the results of \cite{ehh}, where the reaction $f$ was of KPP type. First, we note that the symmetry assumption on $q$ was not needed in the KPP case studied in \cite{ehh}. Also, in \cite{ehh} the cones which appear in the conditions at $\pm\infty$ can have different angles
which were denoted by $\alpha$ and $\beta$.  A main reason leading to these differences is that in the KPP case there is a range of speeds of the form $[c^*_{A},\infty)$ (resp. $[c^*_{B},\infty)$) rather than a unique speed, 
 where $c^*$ denotes the minimal KPP speed of propagation.  This fact allowed the following construction in \cite{ehh}: 
 for a given $c\ge c^*$, there exist $(c_\alpha, \varphi_{\alpha})$ and $(c_\beta, \varphi_{\beta})$  and such that
\begin{equation}\label{subc}
\displaystyle{c=\frac{c_\alpha}{\sin\alpha}=\frac{c_\beta}{\sin\beta}\ge c^*. }
\end{equation}
In this present work,   the  speeds $c_{A}$ and $c_{B}$ are unique  as $f$ is of type \eqref{cf}. Moreover, the KPP type nonlinearity considered in \cite{ehh} is concave on the interval $[0,1]$ while this is not the case for a `combustion' type nonlinearity \eqref{cf} (due to the ignition temperature $\theta$). The concavity of the KPP made the construction of a supersolution that obeys the conical conditions at $\pm\infty$ easier than what we will have in the present work.  The uniqueness of the speed in the `combustion' case, and the non-concavity of the nonlinearity $f$ over $[0,1]$,  will be the main differences that  make the construction of desired solutions more involved than in \cite{ehh}. }
\end{remark}

\section{Proofs}

\subsection{Proof of existence in Theorem \ref{existence}}
We will connect the conical-fronts problem to planar-pulsating  fronts through a change of variables. We denote by 
\begin{equation}\label{components}
\phi_{1}(x,y):=\varphi(x,x\cos\alpha+y\sin\alpha)~\hbox{ and }~\phi_{2}(x,y):=\psi(x,-x\cos\alpha+y\sin\alpha),
\end{equation}
where $\varphi$ and $\psi$ are the unique solutions to \eqref{aux1} and \eqref{aux2} respectively. We will construct a solution to \eqref{cequation} that satisfies the limiting conditions \eqref{concon} via Perron-type  methods  introduced in \cite{amann1} and  Noussair \cite{Noussair} for instance. We start by building a subsolution to the conical problem.
\begin{lemma}[Subsolution]\label{subsolution}
Let \begin{equation}\label{subc}
\displaystyle{c=\frac{c_{A,q\sin\alpha,f}}{\sin\alpha}=\frac{c_{B,q\sin\alpha,f}}{\sin\alpha},}
\end{equation} where $c_{A,q\sin\alpha,f}$ is the unique speed of the pulsating traveling front solving \eqref{aux1}
and let \begin{equation}
\underline{\phi}(x,y):=\max\{\phi_{1}(x,y),\phi_{2}(x,y)\}, ~~(x,y)\in \R^{2},
\end{equation}
where \begin{equation}
\phi_{1}(x,y):=\varphi(x,x\cos\alpha+y\sin\alpha)~\hbox{ and }~\phi_{2}(x,y):=\psi(x,-x\cos\alpha+y\sin\alpha),
\end{equation}
and $\varphi$ (resp. $\psi$) is the solution to \eqref{aux1} (resp. \eqref{aux2}). 
Then $\underline{\phi}$ is a subsolution of \eqref{cequation} with the conditions \eqref{concon}. 
\end{lemma}
\begin{proof} Note that $\phi_{1}$, defined in \eqref{components} by $\phi_{1}(x,y):=\varphi(x,x\cos\alpha+y\sin\alpha),$ satisfies  
\begin{align*}
\Delta \phi_1(x,y)+(q(x)-c)\partial_y\phi_1(x,y)+f(\phi_1)\\
\hbox{div}(A\nabla\varphi)+(q(x)-c)\sin\alpha\partial_Y\varphi+f(\varphi)=0
\end{align*}
for all  $(x,y)\in \mathbb R^2$, where the quantities involving $\varphi$ are taken values at the point $(x,x\cos\alpha+y\sin\alpha).$ Also, $\phi_{2}$, defined by $\phi_{2}(x,y):=\psi(x,-x\cos\alpha+y\sin\alpha),$ satisfies  
\begin{align*}
\Delta \psi(x,y)+(q(x)-c)\partial_y\phi_2(x,y)+f(\phi_2)\\=\hbox{div}(B\nabla\psi)+(q(x)-c)\sin\alpha\partial_Y\psi+f(\psi)=0
\end{align*}
for all  $(x,y)\in \mathbb R^2.$ 
Moreover, since $\sin \alpha >0$ when $\alpha\in (0,\pi),$ it follows that  $\lim_{y\rightarrow-\infty}\phi_{1}(x,y)=\lim_{y\rightarrow-\infty}\phi_{2}(x,y)=0$ and $\lim_{y\rightarrow+\infty}\phi_{1}(x,y)=\lim_{y\rightarrow+\infty}\phi_{2}(x,y)=1.$  Then, for \[\underline{\phi}(x,y):=\max\{\phi_{1}(x,y),\phi_{2}(x,y)\},\] we have  \begin{equation}\label{sub.concon}
\displaystyle{\lim_{l\rightarrow-\infty}}\Big(\displaystyle{\sup_{(x,y)\in C^{-}_{\alpha,l}}}\underline\phi(x,y)\Big)= 0\quad \text{and} \quad 
\displaystyle{\lim_{l\rightarrow\infty}\Big(\inf_{(x,y)\in C^{+}_{\alpha,l}}\underline\phi(x,y)\Big)= 1.}\end{equation}
Therefore, the function 
$\underline{\phi}(x,y):=\max\{\phi_{1}(x,y),\phi_{2}(x,y)\}$
is a subsolution of \eqref{cequation} with the limiting conditions \eqref{concon}. \end{proof}
\subsubsection{Supersolution}  

 We choose the planar fronts represented by the functions $\varphi$ and $\psi$, introduced in \eqref{components} above, such that 

\begin{equation}\label{normalized.panar}\varphi(0,0)=\psi(0,0)=\theta.\end{equation}
The choice in \eqref{normalized.panar} is possible because $\varphi$ and $\psi$ are increasing in the second variable, and satisfy the limiting conditions \eqref{aux1} and \eqref{aux2}.
 
\noindent In order to arrive the desired inequality  \[\Delta\overline{\phi}(x,y)+\left(q(x)-c\right)\partial_y\overline \phi(x,y)+f(\overline{\phi}(x,y))\leq 0 ~\hbox{ for all }~(x,y)\in \mathbb{R}^{2},\] we will divide the plane into several regions according to $(x,y)\mapsto\varphi(x,x\cos\alpha+y\sin\alpha)$, $(x,y)\mapsto\psi(x,-x\cos\alpha+y\sin\alpha)$ and their sum $\phi_{1}+\phi_{2}$. This division of the plane will also  clarify our choice of the functions $H$ and $h$ that appear in the nominated supersolution in formula \eqref{def.phibar} below. We set $Y=x\cos\alpha+y\sin\alpha$  and $Y'=-x\cos\alpha+y\sin\alpha$ and denote by

\begin{equation}\label{set.E1}
E_{1}:=\{(x,y)\in \R^{2}| ~ 0<\varphi(x,Y)+\psi(x,Y')\leq \theta\} 
\end{equation}

\[E:=\{(x,y)\in \R^{2}|~ \varphi(x,Y)+\psi(x,Y')\geq \theta\}. \]

Observe that if $(x,y)\in E$ (that is $(\varphi(x,Y)+\psi(x,Y'))\geq \theta$), then as $\phi>0$ and $\psi>0$,  at least one of $\varphi$ and $\psi$ must be greater or equal $\theta/2.$ We then  divide the set $E$ into the two subregions
\begin{equation}\label{sets.E2.E3}
\begin{array}{c}
E_{2}:=\left\{(x,y)\in E,  \quad  \theta \geq \varphi(x,Y)\;\hbox{ or }\;\theta\geq \psi(x,Y')\right\}
\end{array}
\end{equation}

\paragraph{Relations between the sets $E_{1,2}$ and the pulsating traveling fronts $\phi_{1,2}$.} We use the variables $Y$ and $Y'$ to explore  the relation  of the  functions $\varphi$ and $\psi$ to the sets we constructed above.  

We know from Berestycki and Hamel \cite{bh02} that the following limits hold uniformly in $x$:
\begin{equation}\label{limits.planar}
\begin{array}{c}\lim_{Y\rightarrow -\infty}\varphi(x,Y)=0, \ \lim_{Y'\rightarrow -\infty}\psi(x,Y')=0,  \vspace{12pt} \\ \lim_{Y\rightarrow +\infty}\varphi(x,Y)=1, ~\lim_{Y'\rightarrow +\infty}\psi(x,Y')=1
\vspace{12pt} \\ 
\lim_{Y\rightarrow \pm \infty}\partial_{2}\varphi(x,Y)=0\ \hbox{ and }\ \lim_{Y'\rightarrow \pm\infty}\partial_{2}\psi(x,Y')=0.
\end{array} \end{equation}

\begin{figure}[h!]
\begin{center}\includegraphics[scale=0.885]{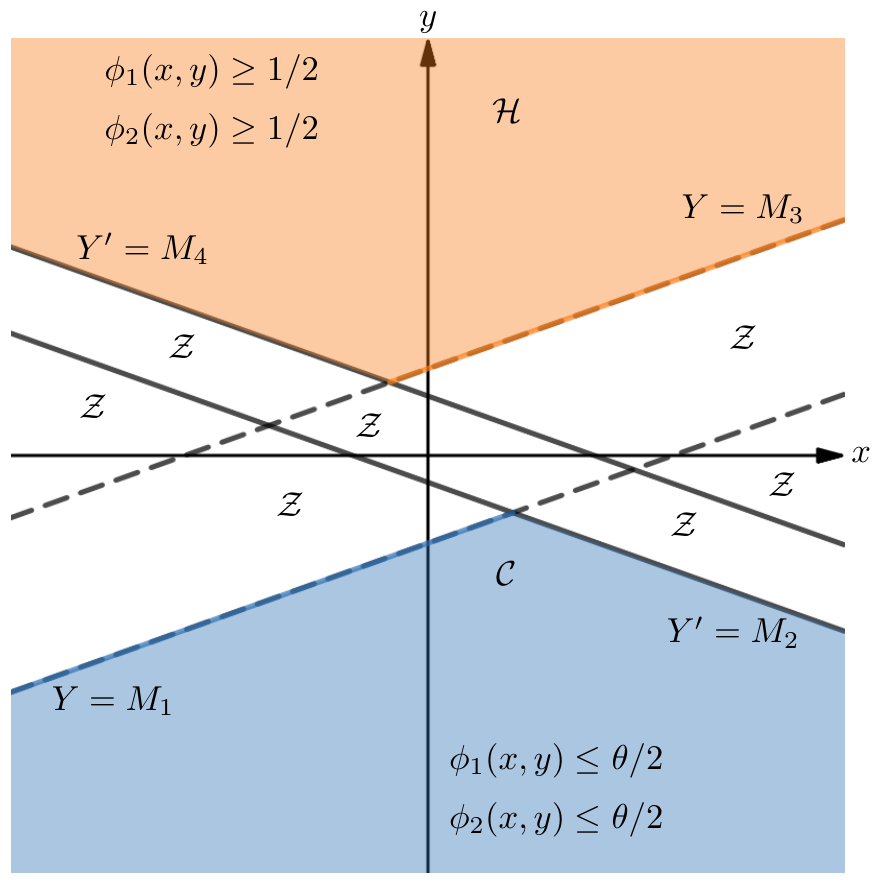}\end{center}
\caption{with $Y=x\cos\alpha+y\sin\alpha$ and $Y'=-x\cos \alpha+y\sin\alpha$, a sketch of Regions $\mathcal{H}$ and $\mathcal{Z}$. The main feature of set $\mathcal{Z}$ is given in \eqref{derivativebound1}}
\end{figure}
\noindent Moreover, $\varphi$ and $\psi$ are increasing in the second variable: $\partial_{2}\varphi(x,Y)>0$ and $\partial_{2}\psi(x,Y')>0$ everywhere in $\R^{2}.$  This allows us to find four constants $M_{1}<0,$ $M_{2}<0$, $M_{3}>0$, $M_{4}>0$  and a $\mu >0$ such that, for all $x\in \R,$
\begin{equation}\label{M1234}
\begin{array}{l}
\varphi(x,Y)\leq \theta/2 \hbox{ when } Y\leq M_{1}, \quad 
\psi(x,Y') \leq \theta/2 \hbox{ when } Y'\leq M_{2},\vspace{8 pt}\\
\varphi(x,Y)\geq 1/2 \hbox{ when } Y\geq M_{3}, \quad
\psi (x,Y')\geq 1/2 \hbox{ when } Y'\geq M_{4},
\end{array}
\end{equation}
\begin{equation}\label{derivativebound1}
\begin{array}{l}
\partial_2\varphi(x, Y)\geq \mu>0 ~  \hbox{ for }~
M_{1}\leq Y \leq M_{3}, \hbox{ and }\vspace{10 pt}\\
\partial_2\psi(x, Y')\geq \mu>0~\hbox{ for  } ~
M_{2}\leq Y' \leq M_{4}.
\end{array}
\end{equation}
Lastly, we introduce the  thresholds ${M_{0}}<M_{1}$ and $M_{0}'<M_{2}$ as follows:
\begin{equation}\label{M0}
\begin{array}{l}
M_{0}:=\sup\left\{y\in\R,~ \varphi(x,x\cos\alpha+y\sin\alpha)\leq \frac{\theta}{4}\hbox{ for all }x\in\R\right\} \hbox{ and }\vspace{7 pt}\\
M_{0}':=\sup\left\{y\in\R,~ \psi(x,-x\cos\alpha+y\sin\alpha)\leq \frac{\theta}{4}\hbox{ for all }x\in\R\right\}.
\end{array}
\end{equation}
Note that $M_{0}$ and $M_{0}'$ are finite due to the monotonicity of $\phi_{1,2}$ in the second argument and their $L$-periodicity in $x.$  This also allows us to find $\mu_{0}>0$ such that 
\begin{equation}\label{derivativebound2}
\begin{array}{l}
\partial_2\varphi(x, Y)\geq \mu_{0}>0 ~  \hbox{ for }
M_{0}\leq Y \leq M_{1}\hbox{ and }\vspace{10 pt}\\
\partial_2\psi(x, Y')\geq \mu_{0}>0~\hbox{ for  } 
M_{0}'\leq Y'\leq M_{2}.
\end{array}
\end{equation}

\noindent Now we define the sets $\mathcal{C},$ $\mathcal{H}$ and $\mathcal{Z}$ by

\begin{equation}\label{setZandCandH}
\begin{array}{c}
\mathcal{H}:=\left\{(x,y)\in\R^{2}~\text{ such that }~Y\geq M_{3} \hbox{ and }Y'\geq M_{4}\right\},\vspace{10 pt}\\
\displaystyle{\mathcal{C}:=\left\{(x,y)\in \R^{2}~\text{ such that }~Y\leq M_{1} \hbox{ and }Y'\leq M_{2}\right\}} \hbox{ and } \vspace{10 pt}\\
\mathcal{Z}:=\R^{2}\setminus (\mathcal{H}\cup\mathcal{C}).
\end{array}
\end{equation}

\begin{remark}Note that  $\mathcal{C}\subseteq E_{1}$ and that the inclusion may be strict. This is because the level sets of the pulsating traveling fronts $\phi_{1}$ and $\phi_{2}$ are not necessarily given by ``perfect'' cones with boundaries parallel to straight lines $Y=\hbox{constant}$ or $Y'=\hbox{constant}.$ The shapes of level sets of nonplanar/curved front solutions to equation \eqref{ueq}, with $q=0$, were studied in Hamel and Monneau \cite{HMonneau}, namely Theorem 1.2. In this present work, since the advection term $q$ is nonzero,  the set $\mathcal{Z}$ will be given a special attention in the construction of a super-solution.   
\end{remark}

In order to construct a supersolution $\overline{\phi}$, we will use an auxiliary function $h$ which will be composed with the sum $\phi_{1}+\phi_{2}$ of the two pulsating traveling fronts introduced above. It turns out that the function $h$ should satisfy a second order differential equation in order to produce a supersolution when composed with $\phi_{1}+\phi_{2}$ (this approach is inspired by the work of Tao, Zhu and Zlato\v{s} \cite{zlatosTaoZhu} dedicated to a different problem.) We will study this ODE in the next lemma and prove few properties of its solutions. These properties will play a role in construction a supersolution to \eqref{cequation} with the limiting conditions \eqref{concon}.   

  \begin{lemma}\label{the.function.h} Let $\beta>0$ be a positive number and let $h_{\beta}$ (write $h$ for simplicity) denote the unique solution to the initial value problem
 \begin{equation}\label{IVP}
 \left\{\begin{array}{ll}
 \beta h''(z)+f(h(z))=0 & \hbox{ for }~\frac{\theta}2<z<2,\\
 h(\theta/2)=\theta,&\\
h'(\theta/2)=2.&
 \end{array}\right.
 \end{equation}
 Then the following assertions hold 
 \begin{enumerate}
 \item[(a)] For any $\beta>0,$ the solution $h_{\beta}$ is strictly increasing on the interval $[{\theta}/{2},2].$
  \item[(b)]  For any $\beta>0,$ $h_{\beta}(1)> 1.$
   \item[(c)]  For any $\beta>0,$ $h_{\beta}(2)> 1.$
 \end{enumerate}
 \end{lemma}

 \begin{proof}[{\bf Proof of part (a) of Lemma \ref{the.function.h}.}] Fix $\beta>0.$ We drop the subscript  $\beta$ for simplicity in writing.  We will use the well known sliding method (see \cite{BerNir} and \cite{bh02}, for example) in order to prove that $h$ is increasing on $[\theta/2,2]:$\\
  for $0<\lambda<2-{\theta}/{2},$ we set \[h^{\lambda}(z):=h(z+2-\frac{\theta}{2}-\lambda) \text{ for } z\in (\frac{\theta}{2},\frac{\theta}{2}+\lambda).\]
 It suffices to prove that 
 \begin{equation}\label{h.less.hlambda}h<h^{\lambda} \text{ over } (\frac{\theta}{2},\frac{\theta}{2}+\lambda),\end{equation}
  for all $0<\lambda<2-{\theta}/{2}.$ 
  
  We begin by recording few facts about the function $h$ which solves \eqref{IVP}. 
 First, since $f\equiv  0$ on $\R\setminus [0,1]$ and $f\geq0$ in $[0,1]$, the strong maximum principle applied to \eqref{IVP} yields that $h(z)>0$ for all $z\in ({\theta}/{2},2).$
 
 \noindent The nonlinearity $f$ is Lipschitz-continuous, so the solution $h$ of the elliptic differential equation \eqref{IVP} is of class $C^{2}$ on $({\theta}/{2},2).$ Knowing that $h'({\theta}/{2})=2>0,$ and that $h'$ is continuous on $[{\theta}/{2},2],$ it then follows that $h'>0$ in an open neighbourhood of $z={\theta}/{2}.$

We can now launch the sliding argument. Let us define 
\begin{equation}\label{sup.sliding}\lambda^{*}:=\sup\{\lambda\in[0,2-{\theta}/{2}]\text{ such that }h<h^{\sigma}\text{ on }({\theta}/{2},{\theta}/{2}+\sigma)\text{ for all }\sigma\leq \lambda\}.\end{equation}
The discussion above ($h$ is strictly increasing on the interval $[\theta/2,z_{1})$ shows that the supremum  in \eqref{sup.sliding} is taken over a nonempty set and that \eqref{h.less.hlambda} holds true for small enough $\lambda$). Our goal is then to show that $\lambda^{*}=2-{\theta}/{2}.$ Suppose on the contrary that $\lambda^{*}<2-{\theta}/{2}.$ By continuity, one has $h\leq h^{\lambda^{*}}$ in $[{\theta}/{2},{\theta}/{2}+\lambda^{*}].$ On the other hand, there exist a sequence $\{\lambda_{n}\}_{n}$ such that $\lambda_{n}>\lambda^{*}$ and 
$\lambda_{n}\rightarrow \lambda^*$ as well as a sequence of points $\{z_{n}\}_{n}$ in $\left({\theta}/{2},{\theta}/{2}+\lambda_{n}\right)$ such that $h^{\lambda_{n}}(z_{n})\geq h(z_{n})$ for all $n\in \mathbb N.$ Thus, up to a subsequence, $z_{n}\rightarrow \bar z$, as $n\rightarrow +\infty$, for some $\bar z\in \left[{\theta}/{2},{\theta}/{2}+\lambda^{*}\right].$ Passing to the limit as $n\rightarrow +\infty$, we  conclude that $h(\bar z)=h^{\lambda^{*}}(\bar z).$ We now define the function $U$ by 
\[U(z)=h^{\lambda^{*}}(z)-h(z) \text{ in }\left[ {\theta}/{2},{\theta}/{2}+\lambda^{*}\right].\]
We know that $U\geq 0$ in $\left[{\theta}/{2},{\theta}/{2}+\lambda^{*}\right].$ Moreover, it follows from \eqref{IVP} and from the assumption that  $f$ is Lipschitz that we can find a bounded function $b$ such that 
\[U''(z)+b(z)U(z)=0~\text{ for }~z\in ({\theta}/{2},{\theta}/{2}+\lambda^{*})\] together with 
 \[
  U\left(\frac{\theta}{2}\right)=h(2-\lambda^{*})-h(\theta/2),~
   U(\bar z)=0 ~\hbox{ and }~ U'(\theta/2)=h'(2-\lambda^{*})-2.\]
If the point $\bar z$ is an interior point (i.e. ${\theta}/{2}<\bar z<{\theta}/{2}+\lambda^{*}$) then, by the strong maximum principle, the function $U$ must be identically $0$ in $\left({\theta}/{2},{\theta}/{2}+\lambda^{*}\right).$ This cannot be true: if $U$ were identically $0$ on $(\theta/2,2),$ then by continuity we get that $U(\theta/2)=0$ and hence $h({\theta}/{2})=h(2-\lambda^{*}).$ As $\theta/2<2-\lambda^{*}<2$ and $h''\leq 0$ (with $h''\not\equiv0$), the strong maximum principle yields that $h$ is constant on $(\theta/2,2),$ which is a contradiction.    Thus the point $\bar z,$ where $U$ vanishes, must be equal to ${\theta}/{2}+\lambda^{*}.$ In such case, the equality $h^{\lambda^{*}}(\bar z)=h(\bar z)=h^{\lambda^{*}}(\lambda^{*}+{\theta}/{2})$ leads to $h(2)=h(\lambda^{*}+{\theta}/{2}).$ Our assumption that $\lambda^{*}<2-{\theta}/{2}$ and the strong maximum principle  applied to the differential equation $h''(z)+f(h(z))=0$ force $h$ to be identically equal to a positive constant over the whole interval $\left[{\theta}/{2},2\right]$, which contradicts $h'({\theta}/{2})>0$. Therefore, $\lambda^{*}=2-{\theta}/{2}$ and the proof of part (a) in  our lemma is complete. 
\end{proof}

  \begin{proof}[{\bf Proof of part (b) of Lemma \ref{the.function.h}.}] 
We fix $\beta>0$ and we write $h$ for $h_{\beta}.$ First, we note that the initial conditions on $h$ at $\theta/2$, i.e. $h(\theta/2)=\theta$ and $h'(\theta/2)=2>1,$ and the continuity of $h'$ yield the existence of $\tau>0$ such that 
\begin{equation}\label{step1}h(z)>z+\frac{\theta}{2},~\text{ for }~ \frac{\theta}{2}\leq z
\leq \frac{\theta}{2}+\tau.\end{equation}
 Denote by 
 \[q(z)=z+\frac{\theta}{2},~\text{ for }~  \frac{\theta}{2}\leq z\leq 2.\] 
 We will compare  $h$ to  $ q$ over the interval $[\theta/2,1]$ in order to arrive at the desired result.  To this end, we will use the sliding method, again, on the functions $h$ and $q^{\lambda},$ where $q^{\lambda}$ is defined by 
 \[\forall \,\lambda\in \left(0, 1-\frac{\theta}{2}\right),\quad  q^{\lambda}(z)=q(z+1-\frac{\theta}{2}-\lambda) ~\text{ for all }~ z\in [\theta/2,\theta/2+\lambda].\] 
 It suffices to compare $h$ to $q^{\lambda}$ on $ [\theta/2,\theta/2+\lambda],$ for all $ 0<\lambda< 1-{\theta}/{2}.$ From \eqref{step1}, we see that $h\geq q^{\lambda}$ on $ [\theta/2,\theta/2+\lambda],$ for $0<\lambda \leq\min\left\{\tau, 1-{\theta}/{2}\right\}.$ We set  \[\lambda^{*}:=\sup \left\{\lambda \in  \left(0,1-{\theta}/{2}\right)\text{ such that }h\geq q^{\mu} \text{ in }\left[{\theta}/{2},{\theta}/{2}+\mu\right]\text{ for all }\mu\leq \lambda \right\}.\]
 Thus, $\lambda^{*}\geq \min\left\{\tau, 1-{\theta}/{2}\right\}>0.$ Our goal is to prove that we will always have $h(1)\geq 1$. We claim that 
 \begin{enumerate}
 \item[(i)] either $\lambda^{*}=1-{\theta}/{2}.$ Thus, $h\geq q$ on $(\theta/2,1 )$ and  so $h(1)\geq 1+{\theta}/{2}>1$
 \item[(ii)] or $\lambda^{*}<1-{\theta}/{2}$ while $h(1)\geq 1+{\theta}/{2}>1.$
 \end{enumerate}
 We know that $\lambda^{*}\leq1-{\theta}/{2}.$ If $\lambda^{*}=1-{\theta}/{2},$  then we have $h\geq q$ on $({\theta}/{2}, 1)$ and thus assertion (i) holds.

  Suppose  in what follows that $\lambda^{*}<1-{\theta}/{2}.$ By continuity, we have $h\geq q^{\lambda^{*}}$ in $[\theta/2,\theta/2+\lambda^{*}].$ As in the previous proof, we can build a sequence $\{\lambda_{n}\}$ such that $\lambda_{n}>\lambda^{*}$ and $\lambda_{n}\rightarrow\lambda^{*}$ and a sequence $\{z_{n}\}_{n}$  in $(\theta/2,\theta/2+\lambda_{n})$ such that $h(z_{n})\leq q^{\lambda_{n}}(z_{n}).$ Thus, up to a subsequence,  $z_{n}\rightarrow\bar{z}$ as $n\rightarrow\infty,$ for some $\bar z\in [\theta/2,\theta/2+\lambda^{*}].$ Then, passing to the limit as $n\rightarrow+\infty,$ we get $h(\bar z)=q^{\lambda^{*}}(\bar z).$ Now let \[Q(z):=h(z)-q^{\lambda^{*}}(z).\]
 We know that $Q\geq 0$ in $[\theta/2,\theta/2+\lambda^{*}]$ and $Q(\bar z)=0.$ Moreover, the function $Q$ satisfies an ODE of the form 
  \begin{equation}\label{equationbyQ}\begin{array}{l}
  \beta Q''(z)+B(z)Q(z)=0\text{ for }z\in (\theta/2,\theta/2+\lambda^{*}),\\
Q(\theta/2)=\theta+\lambda^{*}-1-\frac{\theta}{2}\geq \lambda^{*}+\frac{\theta}{2}-1,\\
Q(\bar z)=0,
\end{array}
\end{equation}
  where $B(z)$ is obtained from the fact that $f$ is Lipschitz.

  If $\bar z$ is an interior point, i.e. ${\theta}/{2}<\bar z< \lambda^{*}+{\theta}/{2},$ then we appeal to \eqref{equationbyQ} and the strong maximum principle to obtain  that $Q\equiv 0$ on    $(\theta/2,\theta/2+\lambda^{*})$ or equivalently $h\equiv q^{\lambda^{*}}.$ This leads to $h'=(q^{\lambda^{*}})' \equiv 1$ and  contradicts the fact that $h'({\theta}/{2})=2.$  

Let us now inspect the case where $\bar z$ is a boundary point.   If $\bar z=\theta/2$ then \[h(\theta/2)=q^{\lambda^{*}}(\bar z)=\theta=\bar z+\frac{\theta}{2}+1-\frac{\theta}{2}-\lambda^{*}=1+\frac{\theta}{2}-\lambda^{*}.\]
  This yields that $\lambda^{*}=1-{\theta}/{2}$ and a  contradiction is obtained. The only possibility left is that  $\bar z=\lambda^{*}+{\theta}/{2}.$ In such case,   \[h(\lambda^{*}+{\theta}/{2})=q^{\lambda^{*}}(\bar z)=\bar z+{\theta}/{2}+1-{\theta}/{2}-\lambda^{*}=1+\theta/2.\]
As $h$ is increasing and $\lambda^{*}+{\theta}/{2}<1,$ it follows that $h(1)>h(\lambda^{*}+{\theta}/{2})=1+\theta/2>1.$
To summarize, if $\lambda^{*}<1-{\theta}/{2}$ we have $h(1)>1.$ Therefore, in both cases (whether $\lambda^{*}<1-{\theta}/{2}$ or $\lambda^{*}=1-{\theta}/{2}$), we have $h(1)>1.$

  \end{proof}

  \begin{proof}[{\bf Proof of part (c) of Lemma \ref{the.function.h}.}] We know from part (b) that $h(1)>1$ and, from  part (a), we know that $h$ is increasing. Thus, $h(2)>h(1)>1$ and the proof of the lemma is now complete. 
 \end{proof}

\subsubsection*{An extension of the function $h$, the solution to \eqref{IVP}.} 
\begin{definition}We extend the function $h$, whose existence and qualitative properties as a solution to \eqref{IVP} were proved in Lemma \ref{the.function.h} above, to the function $H$ over the interval $[0,2]$ as follows:

\begin{equation}H (z)=\left\{\begin{array}{ll}
2\left(z-\frac{\theta}{2}\right)+\theta& \text{ for }~~ 0<z\leq \frac{\theta}{2},\vspace{10 pt}\\
h(z) &\text{ for }~~\frac{\theta}{2}<z\leq 2.
\end{array}\right. \end{equation}
and
\begin{equation}
H''(z)\leq 0 ~\text{ for all }~0\leq z\leq 2.
\end{equation}
The function $H$ is in the class $C^{2}([0,2])$ and satisfies \begin{equation}\label{properties}H(0)=0 ~\text{ and }~ H(2)=h(2)\geq 1 \text{ (from part (c) of Lemma \ref{the.function.h}}).\end{equation}
\end{definition}

In the following proposition we show that a certain choice of $\beta$ makes  \[\bar \phi(x,y)=H(\varphi(x,Y)+\psi(x,Y'))\] a supersolution of equation \eqref{cequation}. 
\begin{proposition}\label{gen.supersol} Let  \begin{equation}\label{beta.s.t.}0<\beta \leq \min\{4 \mu^{2}\sin^{2}\alpha,\mu_{0}^{2}\sin^{2}\alpha\},\end{equation} where $\mu$ and $\mu_{0}$ are the positive constants defined in \eqref{derivativebound1} and \eqref{derivativebound2} above. Let $h:=h_{\beta}$ be the solution of the corresponding initial value problem \eqref{IVP} {\rm(}i.e. the solution to \eqref{IVP} for $\beta$ satisfying \eqref{beta.s.t.}{\rm)}. Then  the function \begin{equation}\label{def.phibar}
\bar \phi(x,y):=H(\varphi(x,Y)+\psi(x,Y'))\end{equation} is a supersolution to equation \eqref{cequation}.
\end{proposition}

\begin{proof}[Proof of Proposition \ref{gen.supersol}]
First, we note that 
$\overline{\phi}$ satisfies the following limiting conditions \[\lim_{l\rightarrow+\infty}\inf_{(x,y)\in C_{\alpha,l}^{+}}\overline\phi (x,y)=H(2)=h(2)\geq 1 \text{ (from Part (c) in Lemma \ref{the.function.h})}\] and 
 \[\lim_{l\rightarrow-\infty}\sup_{(x,y)\in C_{\alpha,l}^{-}}\overline\phi (x,y)=H(0)=0. \]
 
 \noindent Now we compute
$$\begin{array}{l}
\Delta\overline{\phi}(x,y)=H'(\varphi+\psi)\left[\nabla\cdot(A\nabla \varphi)+\nabla\cdot(B\nabla \psi)\right]+\vspace{7 pt}\\
H ''(\varphi+\psi)\left[(\partial_1\varphi+\partial_1\psi+\cos\alpha\partial_2\varphi+\cos\alpha\partial_2\psi)^2+\sin^2\alpha(\partial_2\varphi+\partial_2\psi)^2\right],
\end{array}$$ 
and $$\left(q(x)-c\right)\partial_y\overline \phi(x,y)=(q(x)-c)\sin\alpha H'(\varphi(x,Y)+\psi(x,Y'))[\partial_2\varphi+\partial_2\psi].$$
Thus \begin{equation}\label{phibar}
\begin{array}{l}
\Delta\overline{\phi}(x,y)+\left(q(x)-c\right)\partial_y\overline \phi(x,y)+f(\overline{\phi})=\vspace{9 pt}\\
f(H(\varphi+\psi))-H'(\varphi+\psi)[f(\varphi)+f(\psi)]~+\vspace{9 pt}\\
H''(\varphi+\psi)\left[(\partial_1\varphi+\partial_1\psi+\cos\alpha\partial_2\varphi+\cos\alpha\partial_2\psi)^2+\sin^2\alpha(\partial_2\varphi+\partial_2\psi)^2\right].
\end{array} 
\end{equation}

\noindent Since $H''\leq0$ in $[0,2]$, \eqref{phibar} yields \begin{equation}\label{phibar1}
\begin{array}{c}
\Delta\overline{\phi}(x,y)+\left(q(x)-{c}\right)\partial_y\overline \phi(x,y)+f(\overline{\phi})\vspace{7 pt}\\
\leq  f(H(\varphi+\psi))-H'(\varphi+\psi)[f(\varphi)+f(\psi)]\vspace{7 pt}\\
+H''(\varphi+\psi)\left[\sin^2\alpha(\partial_2\varphi+\partial_2\psi)^2\right].
\end{array} 
\end{equation}
 
\noindent The following is to show that the properties of $h$ mentioned in Lemma \ref{the.function.h} together with those satisfied by the pulsating fronts $\phi_{1,2}$ are sufficient to make the right hand side of \eqref{phibar1} nonpositive  everywhere in the plane $\R^{2}.$ Let $(x,y)\in \R^{2}$ and recall that 
 \[\R^{2}= \mathcal{C}\cup\mathcal{Z}\cup\mathcal{H},\]
where $\mathcal{C},~\mathcal{Z},$ and $\mathcal{H}$ are as defined in \eqref{setZandCandH} above.

\paragraph{Case 1: $(x,y)\in \mathcal{C}$.} Here we have $Y\leq M_{1}<0$ and $Y'\leq M_{2}<0.$ Hence, 
\[\phi_{1}(x,y)+\phi_{2}(x,y)\leq \theta~\text{ from \eqref{M1234}},\] and so \[f(\varphi(x,Y))=f(\psi(x,Y'))=0.\]
Note  that $f(h(\varphi+\psi))$ is {\em not} necessarily equal to zero {\em everywhere} in $\mathcal{C}$ as we only know that $\phi_{1}+\phi_{2}\leq \theta$ (recall that $f\equiv0$ on $[0,\theta]$ and $f>0$ on $(\theta,1)$). In order to ensure that the right hand side of \eqref{phibar1} is nonpositive we have to extract more information from the term $H''(\varphi+\psi)\left[\sin^2\alpha(\partial_2\varphi+\partial_2\psi)^2\right]$ when $(x,y)\in \mathcal{C}.$ We distinguish two subcases: 

 \paragraph{\textit{Case 1.a: $(x,y)\in \mathcal{C}$ and $H(\varphi(x,Y)+\psi(x,Y'))\leq \theta$}.} 
 In this case, the terms \[f(H(\varphi+\psi))~\text{ and }~h'(\varphi+\psi)[f(\varphi)+f(\psi)]\] in \eqref{phibar1} both vanish. Therefore, as $H''\leq 0$ on $[0,2],$ \eqref{phibar1} yields that 
\[\Delta\overline{\phi}(x,y)+\left(q(x)-c\right)\partial_y\overline \phi(x,y)+f(\overline{\phi})
\leq H''(\varphi+\psi)\left[\sin^2\alpha(\partial_2\varphi+\partial_2\psi)^2\right]\leq 0.\]

\paragraph{\textit{Case 1.b: $(x,y)\in \mathcal{C}$ while $\theta<H(\varphi(x,Y)+\psi(x,Y'))<1$}.}
 For such $(x,y)$ we have  \[f(H(\varphi(x,Y)+\psi(x,Y')))>0\] because  $\theta<H(\varphi(x,Y)+\psi(x,Y'))=h(\varphi(x,Y)+\psi(x,Y'))<1.$ By Lemma \ref{the.function.h}, the function  $h$ is strictly increasing on $[0,2].$  Part (b) of Lemma \ref{the.function.h} leads to \[h^{-1}(1)>\varphi(x,Y)+\psi(x,Y')>h^{-1}(\theta)>{\frac{\theta}{2}}>0.\]
 The latter inequality implies that either $\varphi(x,Y)>\frac{\theta}{4}>0$ or $\psi(x,Y')>\frac{\theta}{4}$. Without loss of generality, we assume that $\theta\geq \varphi(x,Y)>\frac{\theta}{4}>0$ is what holds (if not, the same argument can be followed by using $\psi(x,Y')>\frac{\theta}{4}$) with $(x,y)\in \mathcal{C}$. By \eqref{M0}, while $(x,y)\in \mathcal{C},$  we  must have \[M_{1}\geq Y:=x\cos\alpha+y\sin\alpha\geq M_{0},\]
 in which case the derivative bound \eqref{derivativebound2} is valid on $\partial_{2}\varphi.$
 Therefore, keeping in mind that $H''\leq 0$, the right hand side of \eqref{phibar1} can be bounded above  as follows \[\begin{array}{l}
 f(H(\varphi+\psi))-H'(\varphi+\psi)[f(\varphi)+f(\psi)] +H''(\varphi+\psi)\left[\sin^2\alpha(\partial_2\varphi+\partial_2\psi)^2\right]\vspace{7 pt}\\
 = f(h(\varphi+\psi))-h'(\varphi+\psi)[f(\varphi)+f(\psi)] +h''(\varphi+\psi)\left[\sin^2\alpha(\partial_2\varphi+\partial_2\psi)^2\right]\vspace{7 pt}\\
 \leq  f(h(\varphi+\psi))+(\sin^{2}\alpha) \mu_{0}^{2}h''(\varphi+\psi)~\text{ as }~[f(\varphi)+f(\psi)]=0\vspace{7 pt}\\
 \leq f(h(\varphi+\psi))+\beta h''(\varphi+\psi);\text{ provided that }\beta\leq \mu_{0}^{2}\sin^{2}\alpha\vspace{7 pt}\\
 =0,
 \end{array}\]
 where we have used the fact  $\partial_{2}\psi>0$ in $\R^{2}$.

\paragraph{Case 2: $(x,y)\in \mathcal{Z}$.} The choices made in \eqref{derivativebound1} guarantee that $\partial_{2}\phi(x,Y)\geq \mu \text { and }\partial_{2}\psi(x,Y')\geq \mu$ whenever $(x,y)\in\mathcal {Z}.$  Then, as $H'=h'>0,$
the right hand side of \eqref{phibar1}  can be bounded above as
\begin{equation*}
\begin{array}{l}
f(H(\varphi+\psi))-H'(\varphi+\psi)[f(\varphi)+f(\psi)] +H''(\varphi+\psi)\left[\sin^2\alpha(\partial_2\varphi+\partial_2\psi)^2\right] \vspace{7pt}\\
\leq f(H(\varphi+\psi))+\beta H''(\varphi+\psi)\text{ provided that }\beta \leq 4\mu^{2}\sin^{2}\alpha \vspace{7pt}\\
=f(h(\varphi+\psi))+\beta h''(\varphi+\psi)\text{ provided that }\beta \leq 4\mu^{2}\sin^{2}\alpha \vspace{7pt}\\
 \leq 0.
\end{array}
\end{equation*}

\paragraph{Case 3:  $(x,y)\in \mathcal{H}.$} In this case we have $Y\geq M_{3}$ and $Y'\geq M_{4}$ and 
 \begin{equation}\label{sum.above.theta}\varphi(x,Y)+\psi(x,Y')\geq1.
 \end{equation} 
 Since $h$ is increasing, then \eqref{sum.above.theta} and Part (b) of Lemma \ref{the.function.h} yield \[H(\varphi(x,Y)+\psi(x,Y'))=h(\varphi(x,Y)+\psi(x,Y'))\geq h(1)\geq 1~\hbox{ for all }~(x,y)\in \mathcal{H};\]  and hence, the term $f(h(\varphi(x,Y)+\psi(x,Y')))\equiv 0$ when  $(x,y)\in\mathcal{H}.$ Again using $h''\leq 0$ and $h'>0$ on $[0,2]$ we obtain \[-h'(\varphi+\psi)[f(\varphi)+f(\psi)]+h''(\varphi+\psi)\left[\sin^2\alpha(\partial_2\varphi+\partial_2\psi)^2\right]\leq0.\]
Looking at the right hand side  of \eqref{phibar1}, we are now able to conclude that \[\text{for all }(x,y)\in \mathcal{H},\quad\Delta\bar\phi(x,y)+(q(x)-{c})\partial_{y}\bar\phi(x,y)+f(\bar\phi (x,y))\leq0.\]
The proof of Proposition \ref{gen.supersol} is now complete.
\end{proof}

\subsection*{Proof of the existence result in Theorem \ref{existence}}
The existence of a solution follows from the existence of a supersolution and a subsolution. A supersolution to \eqref{cequation} with the conical conditions \eqref{concon} is constructed via the function $\overline{\phi}$ in Proposition \ref{gen.supersol} above. The subsolution is the function $\underline{\phi}$ constructed in Lemma \ref{subsolution} above.

\section{Proof of uniqueness and monotonicity}
Before stating the comparison principles that will be the main tool in proving the monotonicity of the solution,  let us introduce some notations and assumptions that we need in the following statements:

For each $l\in\mathbb R,$ $\alpha,\beta\in(0,\pi)$, we consider $A(x,y)=(A_{ij}(x,y))_{1\leq i,j\leq N}$ as a symmetric $C^{1,\delta}\big(\overline{C_{\alpha,l}^+}\big)$ matrix field satisfying
\begin{eqnarray}\label{cA}\left\{\begin{array}{l}
\exists \,0<\alpha_1\leq\alpha_2, \; \forall(x,y)\in\overline{C_{\alpha,l}^+},\;\forall \xi\in\mathbb{R}^2, \vspace{4 pt}\\
\displaystyle{\alpha_1|\xi|^2 \leq\sum_{1\leq i,j\leq 2}A_{ij}(x,y)\xi_i\xi_j\leq\alpha_2|\xi|^2}.\end{array}\right.
\end{eqnarray}
The set
$$\begin{array}{ll}
\partial C_{\alpha,l}^+:=&\Big\{(x,y)\in\mathbb R^2,~y=-x\cot\alpha+l\hbox{ when }x\geq0,\vspace{3pt}\\
&\quad\quad\quad\quad\hbox{and }y=x\cot\alpha+l\hbox{ when }x\leq0\Big\}\end{array}$$
denotes the boundary of the subset $C_{\alpha,l}^+$ which was introduced in Definition \ref{cones}, and
$$\hbox{dist}\left((x,y);\partial C_{\alpha,l}^+\right)$$
stands for the Euclidean distance from $(x,y)\in\mathbb R^2$ to the boundary $\partial C_{\alpha,l}^+.$
The following is a comparison principle that fits our problem, in a conical setting. This result was proved in \cite{ehh} which is a joint work of the author with F. Hamel and R. Huang.  
\begin{lemma}[\cite{ehh}]\label{comp.principle} 
Let $\alpha \in (0,\pi)$ and $l\in\mathbb{R}$. Let $g(x,y,u)$ be a globally bounded and a globally Lipschitz-continuous function defined in $\overline{C_{\alpha,l}^+}\times\mathbb{R}.$ Assume that $g$ is non-increasing with respect to $u$ in $\mathbb R^2\times[1-\rho,+\infty)$ for some $\rho>0.$ Let $\tilde{q}=\left(q_1(x,y),q_2(x,y)\right)$ be a globally bounded $C^{0,\delta}\big(\overline{C_{\alpha,l}^+}\big)$ vector field $($with $\delta>0)$ and let $A(x,y)=(A_{ij}(x,y))_{1\leq i,j\leq 2}$ be a symmetric $C^{2,\delta}\big(\overline{C_{\alpha,l}^+}\big)$ matrix field satisfying~$(\ref{cA})$.\par
Assume that $\phi^{1}(x,y)$ and $\phi^{2}(x,y)$ are two bounded uniformly continuous functions defined in $\overline{C_{\alpha,l}^+}$ of class $C^{2,\mu}\big(\overline{C_{\alpha,l}^+}\big)$ (for some $\mu>0$). Let  $L$ be the elliptic operator defined by
$$L\phi:=\nabla_{x,y}\cdot(A\nabla_{x,y}\phi)+\tilde{q}(x,y)\cdot\nabla_{x,y}\phi.
$$ and assume that
$$\left\{\begin{array}{ccc}
L\,\phi^{1}+g(x,y,\phi^{1})&\geq&0~\hbox{ in }~C_{\alpha,l}^+,\vspace{5 pt}\\
L\,\phi^{2}+g(x,y,\phi^{2})&\leq&0~\hbox{ in }~C_{\alpha,l}^+,\vspace{5 pt}\\   
\phi^{1}(x,y)\leq \phi^{2}(x,y)&&~~\hbox{on }~\partial C_{\alpha,l}^+ ,\end{array}\right.$$
and that
\begin{equation}\label{as.distance.goes.to.infty}
\displaystyle{\limsup_{\displaystyle{(x,y)\in C_{\alpha,l}^+,\,{\rm dist}\left((x,y);\partial C_{\alpha,l}^+\right)\rightarrow+\infty}}\,[\phi^{1}(x,y)-\phi^{2}(x,y)]}\leq0.
\end{equation}

If $\phi^{2}\geq 1-\rho$ in $\overline{C_{\alpha,l}^+},$ then
$$\phi^{1}\leq\phi^{2}~~\hbox{in}~~\overline{C_{\alpha,l}^+}.$$
\end{lemma}

Changing $s$ into $-s$ in Lemma \ref{comp.principle} leads to the following:
\begin{lemma}\label{comp.principle1} 
Let $\alpha$ and $\beta\in (0,\pi)$ and $l\in\mathbb{R}$. Let $g(x,y,u)$ be a globally bounded and a globally Lipschitz-continuous function defined in $\overline{C_{\alpha,l}^-}\times\mathbb{R}.$ Assume that $g$ is non-increasing with respect to $u$ in $\mathbb R^2\times(-\infty,\delta]$ for some $\delta>0.$ Let $\tilde{q}=\left(q_1(x,y),q_2(x,y)\right)$ be a globally bounded $C^{0,\kappa}\big(\overline{C_{\alpha,l}^-}\big)$ vector field $($with $\kappa>0)$ and let $A(x,y)=(A_{ij}(x,y))_{1\leq i,j\leq 2}$ be a symmetric $C^{2,\kappa}\big(\overline{C_{\alpha,l}^-}\big)$ matrix field satisfying~$(\ref{cA})$.\par
Assume that $\phi^{1}(x,y)$ and $\phi^{2}(x,y)$ are two bounded uniformly continuous functions defined in $\overline{C_{\alpha,l}^-}$ of class $C^{2,\mu}\big(\overline{C_{\alpha,l}^-}\big)$ (for some $\mu>0$). Let  $L$ be the elliptic operator defined by
$$L\phi:=\nabla_{x,y}\cdot(A\nabla_{x,y}\phi)+\tilde{q}(x,y)\cdot\nabla_{x,y}\phi.
$$ and assume that
$$\left\{\begin{array}{ccc}
L\,\phi^{1}+g(x,y,\phi^{1})&\geq&0~\hbox{ in }~C_{\alpha,l}^-,\vspace{5 pt}\\
L\,\phi^{2}+g(x,y,\phi^{2})&\leq&0~\hbox{ in }~C_{\alpha,l}^-,\vspace{5 pt}\\   
\phi^{1}(x,y)\leq \phi^{2}(x,y)&&~~\hbox{on }~\partial C_{\alpha,l}^- ,\end{array}\right.$$
and that
\begin{equation}\label{as.distance.goes.to.infty1}
\displaystyle{\limsup_{\displaystyle{(x,y)\in C_{\alpha,l}^-,\,{\rm dist}\left((x,y);\partial C_{\alpha,l}^+\right)\rightarrow+\infty}}\,[\phi^{1}(x,y)-\phi^{2}(x,y)]}\leq0.
\end{equation}

If $\phi^{1}\leq \delta$ in $\overline{C_{\alpha,l}^-},$ then
$$\phi^{1}\leq\phi^{2}~~\hbox{in}~~\overline{C_{\alpha,l}^-}.$$
\end{lemma}

\subsection{Proof of Theorem \ref{monotonicity}}
We are now ready to give the proof of Theorem \ref{monotonicity} in details.
\begin{proof}
We denote, for $\tau\in\mathbb R,$
$$\phi^{\tau}(x,y):=\phi(x,y+\tau)\ \hbox{for all}\ (x,y)\in\mathbb R^{2}.$$
Suppose that we have proved that $\phi^\tau\geq\phi$ in $\mathbb R^2$ for all $\tau\geq0.$ Since the coefficients $q$ and $f$ are independent of $y$, then for any $h>0$ the nonnegative function $z(x,y):=\phi^h(x,y)-\phi(x,y)$ is a classical solution (due to (\ref{cequation})) of the following linear elliptic equation
$$\Delta_{x,y}z+(q(x)-c)\partial_yz+b(x,y)z=0\hbox{ in }\mathbb R^2,$$
for some globally bounded function $b=b(x,y)$. It then follows from the strong maximum principle that the function $z$ is either identically $0,$ or positive everywhere in $\mathbb R^2$. Due to the conical limiting conditions (\ref{concon}) satisfied by the function $\phi$, we can conclude that the function $z$ can not be identically $0$. In fact, if $z\equiv0$, then $\phi(x,y+h)=\phi(x,y)$ for all $(x,y)\in\mathbb R^2$ with $h>0.$ This yields that $\phi$ is $h-$periodic with respect to $y$, which is impossible from (\ref{concon}). Hence, the function $z$ is positive everywhere in $\mathbb R^2$, and consequently, the function $\phi$ is increasing in~$y$.

From the discussion above, we only need to prove that $\phi^\tau\geq\phi$ for all $\tau\geq0$:\\ Since
$$\displaystyle{\lim_{l\rightarrow-\infty}\Big(\sup_{(x,y)\in C^{-}_{\alpha,l}}\phi(x,y)\Big)= 0} ~\text{ and }~\displaystyle{\lim_{l\rightarrow\infty}\Big(\inf_{(x,y)\in C^{+}_{\alpha,l}}\phi(x,y)\Big)= 1},$$
there exists then $B>0$, large enough such that, $\hbox{for}~ \tau \geq 2B,$
\begin{equation}\label{slide1}
\left\{\begin{array}{l}
\phi(x,y)\leq \theta,~\hbox{for all} ~(x,y)\in C_{\alpha,-B}^-, \vspace{6 pt }\\
\phi^{\tau}(x,y)\geq 1-\rho,~\hbox{for all}~ (x,y)\in C_{\alpha,-B}^+,\vspace{6 pt }\\
\end{array}\right.
\end{equation}
where $\theta$ (the ignition temperature) is the constant appearing in condition \eqref{cf} on the combustion nonlinearity $f$ and $\rho$ is a  constant we choose such that $1-\rho>\theta$ (recall that $\theta<1$). 

We note that for $\tau\geq 2B,$ 
\[\phi(x,y)\leq \phi^{\tau}(x,y)~\hbox{for all}~ (x,y)\in \partial C_{\alpha,-B}^-\,.\]
Thus, appying Lemma \ref{comp.principle1} to the functions $\phi^1:=\phi$ (notice that $\phi$ is actually at least of class $C^{2,\mu}(\mathbb{R}^2)$ for all $0<\mu\le 1$ from the elliptic regularity theory) and $\phi^2:=\phi^\tau$ with $\tau\ge 2B$ while taking $\delta=\theta$, $A=I$, $g=f$ (which is nonincreasing near $0$), $\tilde{q}(x)=\left(0,q(x)-c\right)$ for all  $x\in \mathbb R$ and $l\,=\,-B,$ we obtain that
\begin{equation}\label{on.lower.cone}
\forall\,\tau\geq2B,~~ \phi(x,y)\leq \phi(x,y+\tau)~\hbox{ for all }~(x,y)\in C_{\alpha,-B}^-.
\end{equation}
We turn now to compare $\phi$ to $\phi^{\tau}$ on $C_{\alpha,-B}^+.$ For $\tau\geq 2B,$ we have $\phi^{\tau}\geq 1-\rho$ in $C_{\alpha,-B}^+$ (see \eqref{slide1}) and $\phi^{\tau}(x,y)\geq 1-\rho>\theta \geq \phi(x,y)$ on $\partial C_{\alpha,-B}^+.$ Thus, by Lemma \ref{comp.principle}, we get 
\begin{equation}\label{on.upper.cone}
\forall\,\tau\geq2B,~~  \phi(x,y+\tau)\geq \phi(x,y)~\hbox{ for all }~(x,y)\in C_{\alpha,-B}^+.\end{equation}
From \eqref{on.lower.cone} and \eqref{on.upper.cone} we get that the inequality holds everywhere in $\R^{2}.$ That is, 
\begin{equation}\label{everywhere}
\forall\,\tau\geq2B,~~  \phi(x,y+\tau)\geq \phi(x,y)~\hbox{ for all }~(x,y)\in \R^{2}.\end{equation}

Let us now decrease $\tau$ and set
$$\tau^{*}=\inf\left\{\tau>0,\phi(x,y)\leq\phi^{\tau'}(x,y)\ \hbox{for all}\ \tau'\geq\tau\ \hbox{and for all}\ (x,y)\in\mathbb R^2\,\right\}.$$
The proof of the theorem will be complete once we prove that $\tau^*=0$. We argue by contradiction and assume  that $\tau^*>0$. First, we  note that $\tau^{*}\leq2B$ and, by continuity, we have $\phi\leq\phi^{\tau^*}~\hbox{in}~\mathbb{R}^2.$ Denote by 
$$S:=C_{\alpha,-B}^+\setminus C_{\alpha,B}^{+}$$
the slice located between the ``lower cone'' $C_{\alpha,-B}^{-}$ and the ``upper cone'' $C_{\alpha,B}^+$. Then, for the value of $\displaystyle{\sup_{(x,y)\in \overline{S}}\left(\phi(x,y)-\phi^{\tau^*}(x,y)\right)}$, the following two cases may occur.\hfill\break

\paragraph{\textit{Case 1:}} suppose that
$$\displaystyle{\sup_{(x,y)\in \overline{S}}\left(\phi(x,y)-\phi^{\tau^*}(x,y)\right)<0}.$$
Since the function $\phi$ is (at least) uniformly continuous, there exists $\varepsilon>0$ such that $0<\varepsilon<\tau^*$ and the above inequality holds for all $\tau\in[\tau^*-\varepsilon,\tau^*]$. Then, for any $\tau$ in the interval~$[\tau^*-\varepsilon,\tau^*]$, due to (\ref{on.lower.cone}) and the definition of $S$, we get that
$$\phi(x,y)\leq \phi^\tau(x,y)\hbox{ over }\overline{C_{\alpha,B}^{-}}.$$
Hence, $\phi\leq \phi^\tau$ over $\partial C_{\alpha,B}^{+}$. On the other hand, since $\tau\geq \tau^*-\varepsilon>0$ and $\phi\geq 1-\eta$ over~$\overline{C_{\alpha,B}^{+}}$, we have $\phi^\tau\geq1-\eta$ over $\overline{C_{\alpha,B}^{+}}$. Lemma \ref{comp.principle}, applied to $\phi$ and $\phi^{\tau}$ in $C_{\alpha,B}^{+},$ yields that
$$\phi(x,y)\leq\phi^{\tau}(x,y)\ \hbox{for all}\ (x,y)\in\overline{ C_{\alpha,B}^{+}}.$$
As a consequence, we obtain $\phi\leq\phi^{\tau}$ in $\mathbb{R}^2$ which  contradicts the minimality of $\tau^{*}.$ Therefore, case 1 is ruled out.
\paragraph{\textit{Case 2:}} suppose that $$\displaystyle{\sup_{(x,y)\in \overline{S}}\left(\phi(x,y)-\phi^{\tau^*}(x,y)\right)=0}.$$
Then, there exists a sequence of points $\{(x_n,y_n)\}_{n\in\mathbb N}$ in $\overline{S}$ such that
\begin{equation}\label{at x_n,y_n}
\phi(x_n,y_n)-\phi^{\tau^*}(x_n,y_n)\rightarrow 0~\hbox{as}~n\rightarrow+\infty.
\end{equation}

For each $n\in\mathbb{N},$ call $\phi_n(x,y)=\phi(x+x_n,y+y_n)\hbox{ and }\phi^{\tau^*}_n(x,y)=\phi^{\tau^*}(x+x_n,y+y_n),$ for all $(x,y)\in \mathbb R^2.$ From the regularity of $\phi$, and up to extraction of some subsequence, the functions $\phi_{n}$ and $\phi^{\tau^*}_n$ converge in $C^{2}_{loc}(\mathbb R^2)$ to two functions $\phi_{\infty}$ and $\phi^{\tau^*}_\infty$. On the other hand, since $q$ is globally $C^{1,\delta}\left(\mathbb R\right)$ and is $L-$periodic, we can assume that the functions $q_{n}(x)=q(x+x_n)$ converge uniformly in $\mathbb R$ to a globally $C^{1,\delta}\left(\mathbb R\right)$ function~$q_{\infty}$ as $n\rightarrow+\infty.$

For any $(x,y)\in\mathbb R^2$, set $z(x,y)=\phi_\infty(x,y)-\phi^{\tau^{*}}_\infty(x,y)$. The function $z$ is nonpositive because $\phi\leq\phi^{\tau^*}$ in $\mathbb R^{2}$. Moreover, by passing to the limit as $n\rightarrow+\infty$ in (\ref{at x_n,y_n}), we obtain $z(0,0)=0$. Furthermore, since the function $q$ does not depend on $y$, we know that the function $z$ solves the following linear elliptic equation
$$\Delta_{x,y}z+(q_{\infty}(x)-c)\partial_{y}z+b(x,y) z=0\hbox{ in }\mathbb R^2$$
for some globally bounded function $b(x,y)$ (since $f$ is Lipschitz continuous). Then, the strong elliptic maximum principle implies that either $z>0$ in $\mathbb R^2$ or $z=0$ everywhere in~$\mathbb R^2$. In fact, the latter case is impossible because it contradicts with the conical conditions at infinity (\ref{concon}): indeed, since $(x_n,y_n)\in \bar S$ for all $n\in\mathbb N$, it follows from (\ref{concon}) that $\lim_{y\rightarrow+\infty}\phi_{\infty}(0,y)=1$ and $\lim_{y\rightarrow-\infty}\phi_{\infty}(0,y)=0$, whence the function $\phi_{\infty}$ cannot be $\tau^*$-periodic with respect to $y$, with $\tau^*>0$. Thus, we have $z(x,y)>0$ in $\mathbb R^2$. But, that contradicts with $z(0,0)=0$. So, case 2 is ruled out too.

Finally, we have proved that $\tau^*=0$, which means that $\phi\leq\phi^{\tau}$ for all $\tau\geq0$. Then, it follows from the discussion in the beginning of this proof that the function $\phi$ is increasing in $y$. This completes the proof of Theorem \ref{monotonicity}.
\end{proof}

\subsection{Theorem \ref{existence}: Proof of uniqueness of conical fronts up to a shift in $t$}
The proof of uniqueness of solutions, up to a shift, uses the same techniques as those used above in the proof of monotonicity.    We will do it here for the sake of completeness. 

Suppose that $u^{1}(t,x,y)=\Phi^{1}(x,y+c_{1}t)$ and $u^{2}(t,x,y)=\Phi^{2}(x,y+c_{2}t)$ are both solutions to equation \eqref{cequation} with $c=c_{1}$ and $c=c_{2}$ respectively, and that $\Phi^{1}$ and $\Phi^{2}$ satisfy the 
conical limiting conditions 
\begin{equation}\label{limits.phi}
\displaystyle{\lim_{l\rightarrow-\infty}}\Big(\displaystyle{\sup_{(x,y)\in C^{-}_{\alpha,l}}}\Phi^{1}(x,y)\Big)= 0,\quad  \quad 
\displaystyle{\lim_{l\rightarrow\infty}\Big(\inf_{(x,y)\in C^{+}_{\alpha,l}}\Phi^{1}(x,y)\Big)= 1,}\end{equation}
\begin{equation}\label{limits.PHI}
\displaystyle{\lim_{l\rightarrow-\infty}}\Big(\displaystyle{\sup_{(x,y)\in C^{-}_{\alpha,l}}}\Phi^{2}(x,y)\Big)= 0\quad \text{and} \quad 
\displaystyle{\lim_{l\rightarrow\infty}\Big(\inf_{(x,y)\in C^{+}_{\alpha,l}}\Phi^{2}(x,y)\Big)= 1.}\end{equation}

We can assume, without loss of generality, that $c_{1}\leq c_{2}.$ From Theorem \ref{monotonicity}, we know that $\Phi^{1}$ and $\Phi^{2}$ satisfy $\partial_{2}\Phi_{i}>0$ in $\R^{2},$ for $i=1,2$. The functions $\Phi^{i},$ $i=1,2,$ satisfy \begin{equation}\label{Phi^{i}}\Delta\Phi^{i}+(q(x)-c^{i})\partial_{y}\Phi^{i}+f(\Phi^{i})=0 \;\hbox{ for all }\;(x,y)\in\mathbb{R}^2.\end{equation}
As $\Phi^{1}$ is increasing in its second variable, we then have 
\begin{equation}\label{Phi1}
\Delta\Phi^{1}+(q(x)-c^{2})\partial_{y}\Phi^{1}+f(\Phi^{1})=(c^{1}-c^{2})\partial_{y}\Phi^{1}\leq 0 \;\hbox{ for all }\;(x,y)\in\mathbb{R}^2
\end{equation}
while 
\begin{equation}\label{Phi2}
\Delta\Phi^{2}+(q(x)-c^{2})\partial_{y}\Phi^{2}+f(\Phi^{2})= 0 \;\hbox{ for all }\;(x,y)\in\mathbb{R}^2
\end{equation}

The idea is to slide the function $\Phi^{1}$ with respect to $\Phi^{2}.$   First, we note that \eqref{Phi1} holds for $\Phi^{1,\tau}(x,y):=\Phi^{1}(x,y+\tau)$ as the PDE is invariant with respect to translations in the $y$-variable. Then from \eqref{limits.phi} and \eqref{limits.PHI}, there exists $B>0$ large enough such that, $\hbox{for}~ \tau \geq 2B,$
\begin{equation}\label{slide.phi.1}
\left\{\begin{array}{l}
\Phi^{2}(x,y)\leq \theta,~\hbox{for all} ~(x,y)\in C_{\alpha,-B}^-, \vspace{6 pt }\\
\Phi^{1,\tau}(x,y)\geq 1-\rho,~\hbox{for all}~ (x,y)\in C_{\alpha,-B}^+,\vspace{6 pt }\\
\end{array}\right.
\end{equation}
where  $\theta$ (the ignition temperature) is the constant appearing in condition \eqref{cf} on the combustion nonlinearity $f$ and $\rho$ is a  constant we choose such that $1-\rho>\theta$ (recall that $\theta<1$). Thus we have $\Phi^{2}\leq \Phi^{1,\tau}$ on $\partial C^{-}_{\alpha,-B}.$ Applying Lemma \ref{comp.principle1} we then obtain 
\begin{equation}\label{on.lower.cone1}
\forall\,\tau\geq2B,~~ \Phi^{2}(x,y)\leq \Phi^{1}(x,y+\tau)~\hbox{ for all }~(x,y)\in C_{\alpha,-B}^-.
\end{equation}
Now, we compare $\Phi^{2}$ and $\Phi^{1,\tau}$ on $C_{\alpha,-B}^+.$ For $\tau\geq 2B,$ we have $\Phi^{1,\tau}\geq 1-\rho$ in $C_{\alpha,-B}^+$ and $\Phi^{1,\tau}(x,y)\geq 1-\rho>\theta \geq \Phi^{2}(x,y)$ on $\partial C_{\alpha,-B}^+.$ Thus, by Lemma \ref{comp.principle} we get 
\begin{equation}\label{on.upper.cone1}
\forall\,\tau\geq2B,~~ \Phi^{1,\tau}(x,y)= \Phi^{1}(x,y+\tau)\geq \Phi^{2}(x,y)~\hbox{ for all }~(x,y)\in C_{\alpha,-B}^+.\end{equation}
From \eqref{on.lower.cone1} and \eqref{on.upper.cone1} we get that the inequality holds everywhere in $\R^{2}.$ That is, 
\begin{equation}\label{everywhere}
\forall\,\tau\geq2B,~~  \Phi^{1,\tau}(x,y)\geq \Phi^{2}(x,y)~\hbox{ for all }~(x,y)\in \R^{2}.\end{equation}
Now we start to decrease $\tau$ by setting
\begin{equation}\label{define.inf}\tau^{*}=\inf\left\{\tau>0,\Phi^{2}(x,y)\leq\Phi^{1,\tau'}(x,y)\ \hbox{for all}\ \tau'\geq\tau\ \hbox{and for all}\ (x,y)\in\mathbb R^2\,\right\}.\end{equation}
The proof of the uniqueness of solutions, up to a shift, will be complete once we prove that $\tau^*=0$. As in the previous proof, we argue by contradiction and assume  that $\tau^*>0$.  We  note that $\tau^{*}\leq2B$ and, by continuity, we have $\Phi^{2}\leq\Phi^{1,\tau^*}~\hbox{in}~\mathbb{R}^2.$ Denote by 
$$S:=C_{\alpha,-B}^+\setminus C_{\alpha,B}^{+}$$
the slice located between the ``lower cone'' $C_{\alpha,-B}^{-}$ and the ``upper cone'' $C_{\alpha,B}^+$. Then, for the value of $\displaystyle{\sup_{(x,y)\in \overline{S}}\left(\Phi^{2}(x,y)-\Phi^{1,\tau^*}(x,y)\right)}$, the following two cases may occur.\hfill\break
\paragraph{\textit{Case 1:}} suppose that
$$\displaystyle{\sup_{(x,y)\in \overline{S}}\left(\Phi^{2}(x,y)-\Phi^{1,\tau^*}(x,y)\right)<0}.$$
As $\Phi^{1}$ and $\Phi^{2}$ are continuous, there exists $\eta>0$  such that 
$$\displaystyle{\sup_{(x,y)\in \overline{S}}\left(\Phi^{2}(x,y)-\Phi^{1,\tau}(x,y)\right)<0}$$
for any $\tau\in[\tau^{*}-\eta,\tau^{*}].$ Choose any $\tau \in[\tau^{*}-\eta,\tau^{*}]$
and apply Lemma \ref{comp.principle1} to $\Phi^{1,\tau}$ and $\Phi^{2}$  on $C_{\alpha,-B}^{-}$ to conclude that \[\Phi^{2}(x,y)\leq\Phi^{1,\tau}(x,y)~\text{ for all }~(x,y)\in C_{\alpha,-B}^{-}.\]
As $\Phi^{2}\leq \Phi^{1,\tau}$ in $\overline{S},$ we have $\Phi^{2}(x,y)\leq \Phi^{1,\tau}(x,y)$ for $(x,y)\in \partial C^{+}_{\alpha,B}.$ Moreover, since $\partial_{y}\Phi^{1,\tau}>0,$ then it follows from \eqref{slide.phi.1} that $\Phi^{1,\tau}(x,y)\geq 1-\rho$ in $C_{\alpha,B}^{+}.$ We apply the comparison principle in Lemma \ref{comp.principle} to obtain that \[\Phi^{2}\leq \Phi^{1,\tau} \text{ in }C_{\alpha,B}^{+}.\]
Thus, we have \[\Phi^{2}\leq \Phi^{1,\tau} \text{ in }\R^{2}.\]
This however contradicts the definition of $\tau^{*}$ in \eqref{define.inf} as an infimum. Therefore, Case 1 is ruled out and we are left with the following. 

\paragraph{\textit{Case 2:}}  $\displaystyle{\sup_{(x,y)\in \overline{S}}\left(\Phi^{2}(x,y)-\Phi^{1,\tau^*}(x,y)\right)=0}.$

\noindent Then, there exists a sequence of points $\{(x_n,y_n)\}_{n\in\mathbb N}$ in $\overline{S}$ such that
\begin{equation}\label{at x_n,y_n}
\Phi^{2}(x_n,y_n)-\Phi^{1,\tau^*}(x_n,y_n)\rightarrow 0~\hbox{as}~n\rightarrow+\infty.
\end{equation}
For each $n\in\mathbb{N},$ call \[\Phi^{2}_n(x,y)=\Phi^{2}(x+x_n,y+y_n)\hbox{ and }\Phi^{1,\tau^*}_n(x,y)=\Phi^{1,\tau^*}(x+x_n,y+y_n),\] for all $(x,y)\in \mathbb R^2.$ From the regularity of $\Phi^{2}$ and $\Phi^{1}$, and up to extraction of some subsequence, the functions $\Phi^{2}_{n}$ and $\Phi^{1,\tau^*}_n$ converge in $C^{2}_{loc}(\mathbb R^2)$ to two functions $\Phi^{2}_{\infty}$ and $\Phi^{1,\tau^*}_\infty$. On the other hand, since $q$ is globally $C^{1,\delta}\left(\mathbb R\right)$ and is $L-$periodic, we can assume that the functions $q_{n}(x)=q(x+x_n)$ converge uniformly in $\mathbb R$ to a globally $C^{1,\delta}\left(\mathbb R\right)$ function~$q_{\infty}$ as $n\rightarrow+\infty.$ Moreover, by passing to the limit as $n\rightarrow+\infty$ in (\ref{at x_n,y_n}), we obtain $\Phi_{\infty}^{1,\tau^{*}}(0,0)=\Phi_{\infty}^{2}(0,0)$.

Now we return to the variables $(t,x,y)$ and denote by \[z(t,x,y)=\Phi_{\infty}^{1}(x,y+c_{1}t+\tau^{*})-\Phi^{2}_{\infty}(x,y+c_{1}t)\]
It follows that $z(t=0,x=0,y=0)=0$ and, from \eqref{Phi^{i}}, we have 
\begin{equation}\label{Phi2.infty}
\Delta\Phi_{\infty}^{2}+(q(x)-c^{2})\partial_{y}\Phi_{\infty}^{2}+f(\Phi_{\infty}^{2})= 0 \;\hbox{ for all }\;(x,y)\in\mathbb{R}^2
\end{equation}
and 
\begin{equation}\label{Phi1.infty}
\Delta\Phi_{\infty}^{1,\tau^{*}}+(q(x)-c^{1})\partial_{y}\Phi_{\infty}^{1,\tau^{*}}+f(\Phi_{\infty}^{1,\tau^{*}})= 0 \;\hbox{ for all }\;(x,y)\in\mathbb{R}^2.
\end{equation}
Thus,
\begin{equation}\label{zeq}
\partial_{t}z-\Delta z-q(x)\partial_{y}z\leq f(\Phi_{\infty}^{1}(x,y+c_{1}t+\tau^{*}))-f(\Phi^{2}_{\infty}(x,y+c_{1}t))
\end{equation}
Again, since $f$ is Lipschitz-continuous, there exists a bounded function $b(t,x,y)$ such that
\[\partial_{t}z-\Delta z-q(x)\partial_{y}z+b(t,x,y)\,z=0~\hbox{ for all }~(t,x,y)\in \R^{3},\]
with $z(0,0,0)=0.$ The strong parabolic maximum principle, applied to the last PDE, yields that $z\equiv 0$ for all $t\leq 0$ and $(x,y)\in \R^{2}.$  This leads to \[\Phi_{\infty}^{1}(x,y+\tau^{*})=\Phi^{2}_{\infty}(x,y)~\hbox{ for all }~(x,y)\in \R^{2}.\] Putting this into \eqref{Phi1.infty}, we obtain that  $(c^{1}-c^{2})\partial_{y}\Phi_{\infty}^{2}\equiv 0.$ As $\Phi^{2}$ is increasing in $y$, we must then have $c^{2}=c^{1}.$ Now, the strong elliptic maximum principle and the equations \eqref{Phi1} and \eqref{Phi2} yield that $\Phi^{1}(x,y+\tau)\equiv \Phi^{2}(x,y).$ Therefore, 
\[u^{1}(t+\frac{\tau^{*}}{c},x,y)=u^{2}(t,x,y)~\hbox{ for all }~(t,x,y)\in \R^{3},\] where $c=c^{1}=c^{2}.$ This completes the proof.

\end{document}